\documentclass[12pt,reqno]{amsart}
\usepackage[T1]{fontenc} 
\usepackage[margin=1.45in]{geometry}
\usepackage{graphicx}
\usepackage{amssymb,amsmath,amsthm}
\usepackage[all]{xy}
\usepackage{multicol}
\usepackage{amsfonts,amsmath, tikz}
\usetikzlibrary{topaths,calc}
\usepackage{leftidx}
\usepackage{tikz}
\usetikzlibrary{topaths,calc}
\usepackage{mathrsfs}
\usepackage{lscape}
\usepackage{setspace}
\usepackage{wrapfig}

\newtheorem{theorem}{Theorem}[section]
\newtheorem{proposition}[theorem]{Proposition}
\newtheorem{corollary}[theorem]{Corollary}
\newtheorem{lemma}[theorem]{Lemma}

\theoremstyle{definition}
\newtheorem{definition}[theorem]{Definition}
\newtheorem{definition*}{Definition}
\newtheorem{remark}[theorem]{Remark}
\numberwithin{equation}{section}
\newtheorem{Example}[theorem]{Example}

\begin{document}

\title{Poisson Cohomology of a class of Log Symplectic Manifolds} 

\author{Melinda Lanius}
\address{Department of Mathematics, University of Illinois, 1409 West Green Street, Urbana, IL 61801, USA.}
\email{lanius2@illinois.edu}

\begin{abstract} We compute the Poisson cohomology of a class of Poisson manifolds that are symplectic away from a collection $D$ of hypersurfaces. These Poisson structures induce a generalization of symplectic and cosymplectic structures, which we call a k-cosymplectic structure, on the intersection of hypersurfaces in $D$. 
\end{abstract}

\maketitle

\section{\bf Introduction}
Poisson cohomology, a way of parametrizing the deformations of a Poisson structure, is an important invariant in Poisson geometry. However, the computation of Poisson cohomology is quite difficult in general and few explicit results are known (\cite{Dufour}, p.43). One class of manifolds where the Poisson cohomology is known is the case of b-symplectic manifolds. A \emph{b-symplectic manifold}, defined by Victor Guillemin, Eva Miranda, and Ana Rita Pires \cite{Guillemin01}, is a $2n$-dimensional manifold $M$ equipped with a Poisson bi-vector $\pi$ that is non-degenerate except on a hypersurface $Z$ where there exist coordinates such that locally $Z=\left\{x_1=0\right\}$ and 
$$\displaystyle{\pi=x_1\dfrac{\partial}{\partial x_1}\wedge \dfrac{\partial}{ \partial y_1}+\sum_{i=2}^n \dfrac{\partial}{\partial x_i}\wedge \dfrac{\partial}{\partial y_i}.}$$ 
Recently, much work has been done studying the various facets of $b$-symplectic structures (a few select examples include \cite{Frejlich, Guillemin02, Marcut02}). In particular, Ioan M\u{a}rcut and Boris Osorno Torres (\cite{Marcut01}, Prop.1) and Guillemin, Miranda, and Pires (\cite{Guillemin01}, Thm. 30) showed that the Poisson cohomology $H^p_\pi(M)$ is isomorphic to the de Rham cohomology of a specific Lie algebroid, the $b$-tangent bundle, and hence,  $$H^p_\pi(M)\simeq H^p(M)\oplus H^{p-1}(Z).$$
 
In \cite{Lanius}, we introduced an approach to computing Poisson cohomology using the de Rham cohomology of a Lie algebroid, called the rigged algebroid. We will employ this method to study a class of log symplectic manifolds, a generalization of the b-symplectic case as formulated by Marco Gualtieri, Songhao Li, Alvaro Pelayo, and Tudor Ratiu \cite{Gualtieri02}. 

We will consider smooth manifolds M together with a finite set $D$ of smooth hypersurfaces $Z_i\subset M$ that intersect transversely. In other words, $D$ is a smooth normal crossing divisor on $M$. The $b$-tangent bundle over $(M,D)$, called the $\log$ tangent bundle ${}^{log}TM$ in  \cite{Gualtieri02}, is the vector bundle whose smooth sections are the vector fields tangent to $D$, that is $$\displaystyle{\left\{u\in\mathcal{C}^\infty(M;TM):u|_{Z}\in\mathcal{C}^\infty(Z,TZ)\text{ for all }Z\in D \right\}}.$$ This vector bundle is a Lie algebroid with anchor map the inclusion into $TM$ and with bracket induced by the standard Lie bracket on $TM$. 

Let $\tau$ denote a nonempty subset of $\left\{1,\dots,|D|\right\}$. In \cite{Gualtieri02}, they point out that the Lie algebroid cohomology of the $b$-tangent bundle over $(M,D)$ is 
\begin{equation}\label{eq:logcohom}{}^{b}H^p(M)\simeq H^p(M)\oplus\bigoplus_{|\tau|\leq p}H^{p-|\tau|}\big(\bigcap_{t\in \tau}Z_t\big).\end{equation} In the case when $D$ is a single hypersurface $Z\subset M$, we recover the de Rham cohomology of the $b$-tangent bundle originally computed by Rafe Mazzeo and Richard Melrose (\cite{MelroseGreenBook}, Prop. 2.49). 

\begin{definition*}\cite{Gualtieri02} A \emph{log symplectic structure} on $(M,D)$ is a closed  non-degenerate 2-form $\omega$ in the de Rham complex of the $b$-tangent bundle. In other words,  $$\omega\in{}^{b}\Omega^2(M)=\mathcal{C}^\infty\big(M;\wedge^2\big({}^{b}T^*M\big)\big)$$ satisfying $$d\omega=0\text{ and }\omega^n\neq 0.$$ The form $\omega$ induces a map $\omega^\flat$ between the $b$-tangent and $b$-cotangent bundles.\begin{center}$\xygraph{!{<0cm,0cm>;<1cm, 0cm>:<0cm,1cm>::}
!{(1.25,0)}*+{{}^{b}TM}="b"
!{(4.9,0)}*+{{{}^{b}T^*M}}="c"
!{(1.75,.1)}*+{{}}="d"
!{(4.25,.1)}*+{{}}="e"
!{(1.75,-.1)}*+{{}}="f"
!{(4.25,-.1)}*+{{}}="g"
"d":"e"^{\omega^{\flat}}
"g":"f"^{\pi^{\sharp}=(\omega^{\flat})^{-1}}}$\end{center} The inverse map is induced by a bi-vector $\pi\in\mathcal{C}^\infty(M;\wedge^2({}^{b}TM))$.  This bi-vector is called a \emph{log Poisson structure} on $(M,D)$.  

\end{definition*}

Log symplectic manifolds are a broad generalization of $b$-symplectic manifolds as developed in \cite{Guillemin01}. We restrict our attention to a class of structures that satisfy some nice geometric features of $b$-symplectic structures that are lost in the whole of the $\log$ symplectic category. In particular, Victor Guillemin, Eva Miranda, and Ana Rita Pires \cite{Guillemin01} showed that every $b$-symplectic form induces a cosymplectic structure $$(\theta,\eta)\in\Omega^1(Z)\times\Omega^2(Z)$$ on its singular hypersurface $Z$. That is, there exists a pair of closed forms $\theta, \eta$ on $Z$ such that $$\theta\wedge\eta^{n-1}\neq 0$$ where the dimension of $Z$ is $2n-1$. We will consider certain $\log$ symplectic structures that induce a generalization of cosymplectic structures on the intersection of any subset of hypersurfaces in $D$. In particular we would like the induced cosymplectic structures on each $Z\in D$ to intersect `nicely'. 

\begin{Example}\label{motivatingexample} As a motivating example, let us examine some cases on the 4-torus $\mathbb{T}^4$ identified as $\mathbb{T}^2\times\mathbb{T}^2$ with angular coordinates $(a_1,a_2)$ and $(b_1,b_2)$ respectively. Let $D=\left\{Z_1,Z_2\right\}$ where $Z_1$ is the zero set of $\sin(a_1)$ and $Z_2$ is the zero set of $\sin(a_2)$. The $b$-tangent bundle is generated by the vector fields $$\sin(a_1)\dfrac{\partial}{\partial a_1},\hspace{.25ex}\sin(a_2)\dfrac{\partial}{\partial a_2},\hspace{.25ex}\dfrac{\partial}{\partial b_1},\hspace{.25ex}\dfrac{\partial}{\partial b_2}.$$ 

We will look at three symplectic forms on this bundle. As an example of the behavior we desire, consider the $\log$ symplectic forms  $$\omega_I=\dfrac{da_1}{\sin(a_1)}\wedge\dfrac{da_2}{\sin(a_2)}+db_1\wedge db_2\hspace{2ex}$$\text{ and }$$\omega_{II}=\dfrac{da_1}{\sin(a_1)}\wedge db_1-\dfrac{da_2}{\sin(a_2)}\wedge db_2.$$ The corresponding bi-vectors are respectively given by $$\pi_I=\sin(a_2)\sin(a_1)\dfrac{\partial}{\partial a_2}\wedge\dfrac{\partial}{\partial a_1}+\dfrac{\partial}{\partial b_2}\wedge \dfrac{\partial }{\partial b_1}$$
and$$\pi_{II}=\sin(a_1)   \dfrac{\partial }{\partial b_1}\wedge\dfrac{\partial}{\partial a_1}-\sin(a_2)\dfrac{\partial}{\partial b_2}\wedge\dfrac{\partial}{\partial a_2}.$$ \vspace{1ex}

\noindent \begin{minipage}{0.6\linewidth}

\hspace{3ex}In the symplectic foliation induced by $\pi_I$, the submanifold $\left\{a_1=0,a_2=0\right\}$ is a symplectic leaf with symplectic form $db_1\wedge db_2$. On the other hand, $\pi_{II}$ induces a foliation on  $\left\{a_1=0,a_2=0\right\}$ of points. To see this, note that the leaves associated to $\pi_{II}$ on $\left\{a_1=0,a_2=0\right\}$ are given by the distribution $\ker db_1\cap \ker db_2=\left\{0\right\}$.  

\hspace{3ex} Thus we say $\omega_{I}$ induces the symplectic form $db_1\wedge db_2$ on $\left\{a_1=0,a_2=0\right\}$ and that $\omega_{II}$ induces the closed 1-forms $db_1$ and $db_2$ satisfying $db_1\wedge db_2\neq 0.$ We call the pair $(db_1,db_2)$ a $2$-cosymplectic structure (see Definition \ref{kco}). \end{minipage}
\hspace{-8ex}\begin{minipage}{0.4\linewidth}
 
 \hspace{11ex}{\bf $\begin{array}{c} \text{ symplectic foliation}\\ \text{ on }\left\{a_1=0,a_2=0\right\}\end{array}$}
\begin{multicols}{2}
\begin{flushright}$~$\vspace{2ex}

$\pi_I$ \vspace{9ex} 

$\pi_{II}$ \end{flushright}\columnbreak

\includegraphics[scale=.3]{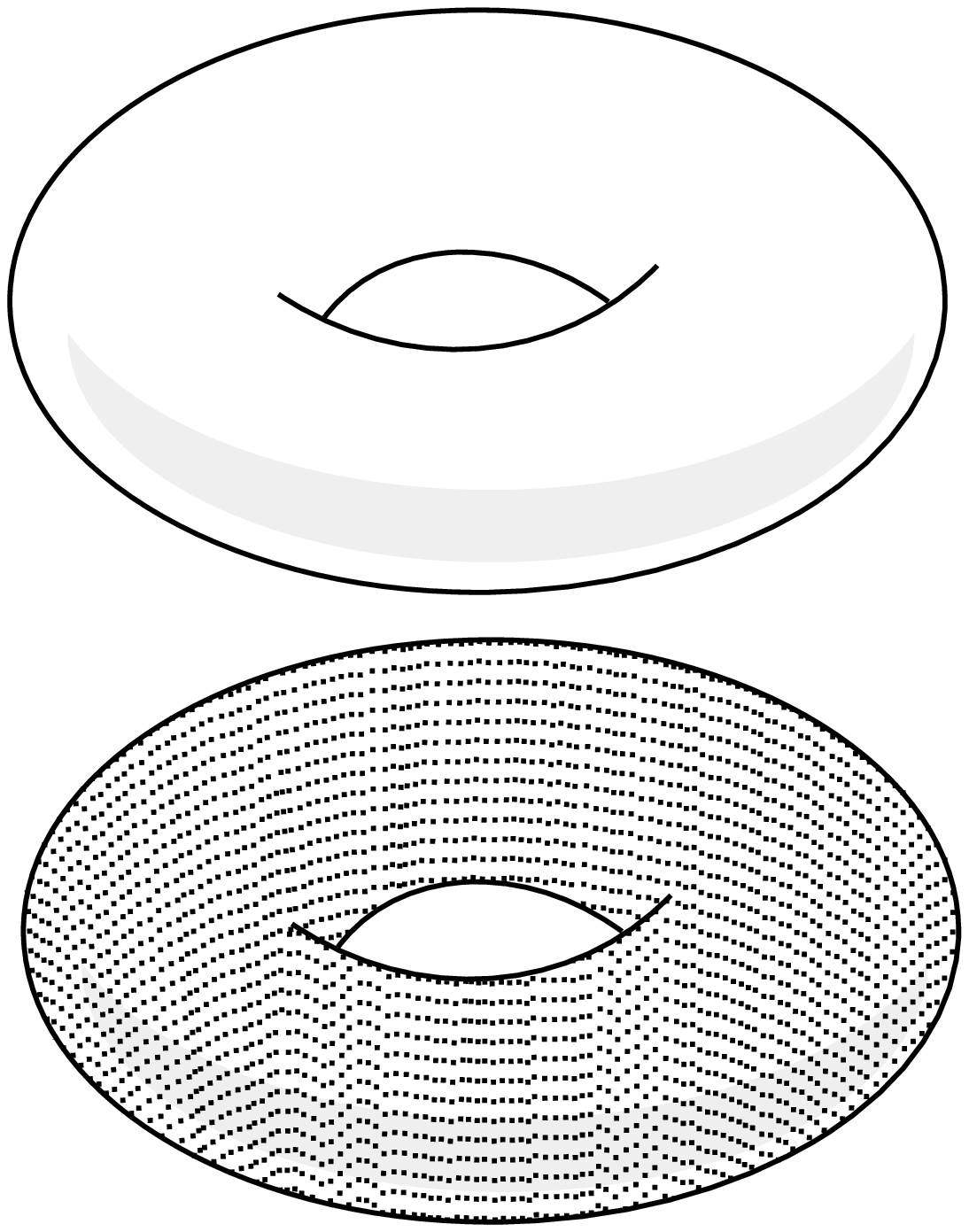} \end{multicols} \end{minipage}\vspace{1ex} 

We are interested in studying $\log$ symplectic forms that induce structures like these at the intersection of elements in $D$. Not all log symplectic forms on the 4-torus will induce a symplectic or $2$-cosymplectic structure. For instance, consider the log symplectic form $$\omega_{III}=\cos(b_1)\omega_I+\sin(b_1)\omega_{II}.$$ The inverse is given by $$\pi_{III}=\cos(b_1)\pi_I+\sin(b_1)\pi_{II}.$$

\noindent \begin{minipage}{0.6\linewidth}

\hspace{3ex} In the case of $\pi_{III}$, the induced folitation on $\left\{a_1=0,a_2=0\right\}$ is not regular. There are two open leaves  $\left\{a_1=0,a_2=0\right\}\setminus\left\{b_1=\pi/2,3\pi/2\right\}.$ At $\left\{b_1=\pi/2,3\pi/2\right\}$, the foliation is given by the distribution $\ker db_1\cap \ker db_2$ and its leaves are points. Because $\cos(b_1)db_1\wedge db_2$ vanishes at $b_1=\pi/2$ and $3\pi/2$, the form $\omega_{III}$ does not induce a global symplectic structure on $\left\{a_1=0,a_2=0\right\}$. 
\end{minipage}
\hspace{-7ex}\begin{minipage}{0.4\linewidth}
 
 \hspace{10ex}{\bf $\begin{array}{c} \text{ symplectic foliation}\\ \text{ on }\left\{a_1=0,a_2=0\right\}\end{array}$}
\begin{multicols}{2}
\begin{flushright}$~$\vspace{2ex}

$\pi_{III}$\end{flushright}\columnbreak

\includegraphics[scale=.3]{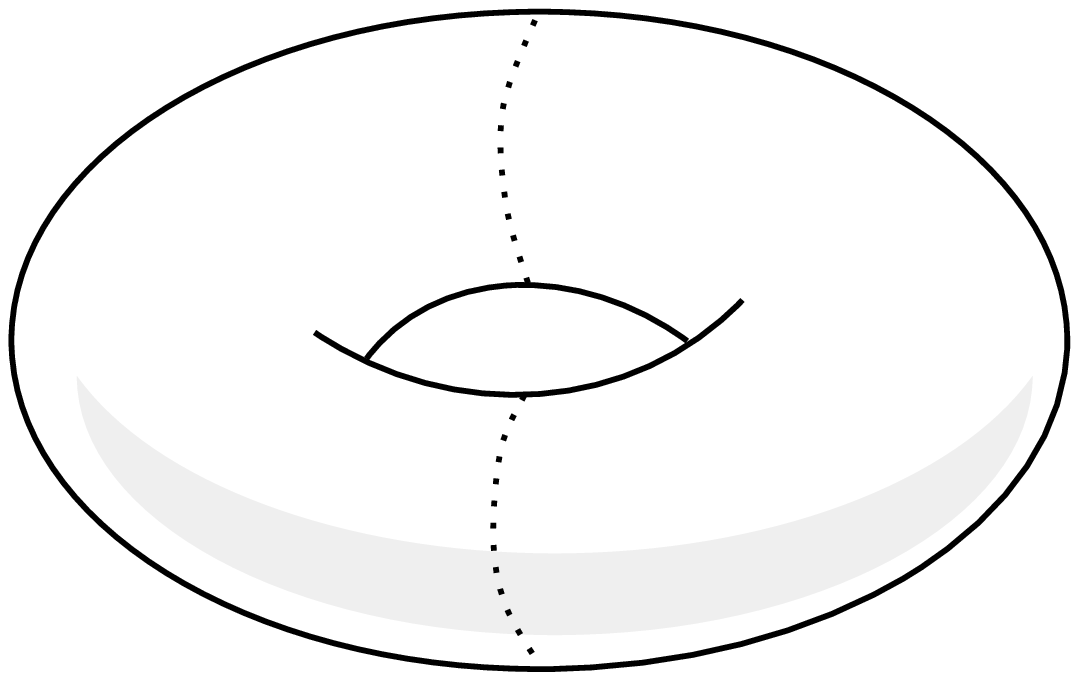} \end{multicols} \end{minipage}\vspace{1ex} 

 Because $\sin(b_1)$ vanishes at $b_1=0$ and $\pi$, the pair   $(\sin(b_1)db_1,\sin(b_1)db_2)$ is not a  2-cosymplectic structure on $\left\{a_1=0,a_2=0\right\}$. Since the foliation on $\omega_{III}$ does not arise from a global structure, such as symplectic or $2$-cosymplectic, on $\left\{a_1=0,a_2=0\right\}$, we will exclude forms such as this from our discussion. \end{Example} 

\begin{definition}\label{kco} A $k$-cosymplectic structure\footnote{The term k-cosymplectic has also been used in classical field theories, with a different meaning, see e.g. \cite{Cappelletti} for details.} on a   $k+2\ell$ dimensional manifold $M$ is a family $(\alpha_i,\beta)$ of $k$ closed one forms $\alpha_i\in\Omega^1(M)$ and a closed two form $\beta\in\Omega^2(M)$ such that

\centerline{$\left(\wedge_{i}\alpha_i\right)\wedge\beta^{\ell}\neq 0.$}\end{definition}

By restricting our attention to log forms satisfying certain cohomological restrictions, we guarantee the existence of such a structure at the intersection of any subset of $D$. 

To any $\log$ symplectic form $\omega$ we can associate a \emph{decomposition} of its cohomology class $[\omega]\in {}^{b}H^2(M)$ in terms of the isomorphism (\ref{eq:logcohom}): 
$$(a,\underbrace{b_1,\dots,b_k}_{b_i\in H^{1}(Z_i)}, \underbrace{c_{1,2},\dots,c_{k-1,k}}_{c_{i,j}\in H^{0}(Z_i\cap Z_j)})\in H^2(M)\oplus\bigoplus_{i} H^{1}(Z_i)\oplus\bigoplus_{i\neq j} H^{0}(Z_i\cap Z_j).$$ 

\begin{definition} Let $\omega$ be a $\log$ symplectic form on a manifold $(M,D)$. The structure $\omega$ is \emph{partitionable} if its $b$-de Rham cohomology class decomposition $$(a,b_1,\dots,b_k,c_{1,2},\dots,c_{k-1,k})$$ satisfies the following conditions. 
\begin{enumerate}
    \item If $b_{s}\neq 0$, then $c_{i,s}=0$ for all $i$. \vspace{1ex}
    \item Consider the inclusions \begin{center}$\xygraph{!{<0cm,0cm>;<1cm, 0cm>:<0cm,1cm>::}
!{(1,1.5)}*+{Z_s\cap Z_t}="b"
!{(0,0)}*+{{Z_s}}="c"
!{(2,0)}*+{{Z_t}}="d"
"b":"c"^{i_s}
"b":"d"_{i_t}}$.\end{center} If $c_{s,t}\neq 0$, then $i^*_sb_{s}=i^*_tb_{t}=0$, and $c_{s,j}=c_{i,t}=0$ for all $j\neq t$, $i\neq s$. 
\end{enumerate}\end{definition} 
 
\begin{remark}\label{partitionremark} Every partitionable type form $\omega$ determines a partition $\Lambda_D$ of the set $D$. By the definition of partitionable, if $b_j=0$ for $Z_j\in D$, then there must exist exactly one element $Z_{j^\prime}\in D$ such that $c_{j,j^\prime}\neq 0$. The tuple $Z_j,Z_{j^\prime}$ forms a pair which we can label $Z_{x_i},Z_{y_i}$. We call this a pair of type $x$ and type $y$. If $b_j\neq 0$ for $Z_j\in D$, then we will label $Z_j$ in the form $Z_{z_i}$. We call these hypersurfaces of type $z$. 
Thus the set $D$ admits a relabeling  $$Z_{x_1},Z_{y_1},\dots,Z_{x_k},Z_{y_k},\hspace{3ex}Z_{z_1},\dots,Z_{z_\ell}.$$  Up to switching the labels $x_i$ and $y_i$, and permuting the set $\left\{1,\dots,k\right\}$ and $\left\{1,\dots,\ell\right\}$, this partition is unique.  

As we saw from $\omega_I$ in Example \ref{motivatingexample}, the intersection $Z_{x_i}\cap Z_{y_i}$ is a leaf of the induced symplectic foliation. The log symplectic form $\omega_{II}$ showed that $Z_{z_i}\cap Z_{z_j}$, when non-empty, has a codimension 2 foliation. Intuitively, intersections of pairs $Z_{x_i},Z_{y_i}$ will be symplectic leaves. The symplectic leaves drop in dimension on intersections of hypersurfaces of type $Z_{z_j}$. \end{remark}

This partition $\Lambda_D$ of $D$ is vital in our computation and statement of the Poisson cohomology. Similar to the b-symplectic case, the Poisson cohomology of partitionable $\log$ symplectic structures will involve the de Rham cohomology of the $b$-tangent bundle, however the two will not always be isomorphic.  
\begin{theorem}\label{theorem01} Let $(M,\pi)$ be a partitionable $\log$ symplectic structure for $$D=\left\{Z_{x_1},Z_{y_1},\dots,Z_{x_k},Z_{y_k},Z_{z_1},\dots,Z_{z_\ell}\right\}.$$ Let $\mathscr{M}$ denote all collections of sets $I,J,K,L$ satisfying $$I,J,K\subseteq\left\{1,\dots,k\right\}, L\subseteq\left\{1,\dots,\ell\right\}$$ $$\text{ such that } I\neq\emptyset\text{ and } I\cap J=I\cap K =\emptyset.$$ 
Set $$m=2|I|+|J|+|K|+|L|$$  and let $v_i$ denote $Z_{x_i}\cap  Z_{y_i}$. 
Then the Poisson cohomology  $H_{\pi}^p(M)$ of $(M,\pi)$ is $${}^{b}H^p(M)\oplus\bigoplus_{\mathscr{M}}H^{p-m}(\bigcap\underbrace{Z_{x_i}\cap Z_{y_i}\cap Z_{x_j}\cap Z_{y_k}\cap Z_{z_\ell}}_{i\in I,~j\in J,~ k\in K,~\ell\in L};\bigotimes\underbrace{|N_{v_i}^*Z_{x_i}|^{-1}\otimes|N_{v_i}^*Z_{y_i}|^{-1}}_{i\in I})$$ for $m\leq p$.  \end{theorem} 

\begin{remark} The factor of ${}^{b}H^p(M)$ appearing in the cohomology is encoding purely topological information about the $b$-tangent bundle over the manifold $M$ and the set $D$. The remaining factors appearing in the cohomology are encoding specific information about the bi-vector $\pi$. \end{remark}

The proof of Theorem \ref{theorem01} appears in Section 3. In Section 2 we discuss aspects of the geometry of partitionable $\log$ symplectic structures and provide examples.  

\vspace{3ex}

{\bf Acknowledgements:} I greatly benefited from attending the 2016 Poisson geometry meeting Gone Fishing held at the University of Colorado at Boulder. Travel support was provided by NSF Grant DMS 1543812 for the Gone Fishing 2016 Conference. I am particularly grateful to Ioan M\u{a}rcut for inspiring this project with his suggestion that I apply my approach for computing Poisson cohomology to other settings. I am grateful to Pierre Albin for carefully reading several versions of this paper. Travel support was provided by Pierre Albin's Simon's Foundation grant \# 317883.

\section{\bf Partitionable log symplectic structures} 

In this section we will discuss various features of partitionable $\log $ Poisson structures. 
\subsection{Local normal forms and $k$-cosymplectic structures} 
\begin{definition} A map $\phi:(M_1, D_1)\to(M_2, D_2)$ is a \emph{$b$-map} if  $$\phi^*:{}^{b}\Omega^1(M_2)|_{M_2\setminus D_2}\to{}^{b}\Omega^1(M_1)|_{M_1\setminus D_1}$$ extends to a map $$\phi^*:{}^{b}\Omega^1(M_2)\to{}^{b}\Omega^1(M_1).$$ Given two log symplectic forms $\omega_1$, $\omega_2$ on $(M,D)$, a log-\emph{symplectomorphism} is a $b$-map $\phi:M\to M$ such that $\phi^*\omega_2=\omega_1$. 
\end{definition} 

As noted in Remark \ref{partitionremark}, given a partitionable $\log$ symplectic form $\omega$ on a manifold $(M,D)$, the cohomological decomposition of $[\omega]$ gives us a partition $\Lambda_D$ of the set $D$ as $Z_{x_1},Z_{y_1},\dots,Z_{x_k},Z_{y_k}, Z_{z_1},\dots,Z_{z_\ell}.$ 

Given a subset $S$ of a divisor $D$, we call  $$X_S=\bigcap_{Z\in S} Z$$ a \emph{maximal intersection} if $X$ is non-empty and if $Z \cap X_S=\emptyset$ for all $Z\in D\setminus S$. 

We can assign a subpartition $\Lambda_S$ to the subset $S$  according to the decomposition $(a,b_1,\dots,b_k,c_{1,2},\dots,c_{k-1,k})$ of $[\omega]\in {}^{b}H^2(M)|_{X_S}$.  In particular, hypersurfaces of type $z$ in $\Lambda_D$ will remain type $z$ in the subpartition $\Lambda_S$. However if there is a hypersurface $Z_{x_i}$ of type $x$ in $S$ and its type $y$ counterpart $Z_{y_i}$ is not in $S$, then $Z_{x_i}$ will have a type $z$ designation in the subpartition $\Lambda_S$.  

We will show that there are hypersurface defining functions, i.e. $x\in\mathcal{C}^\infty(M)$ such that $Z_x=\left\{x=0\right\}$ and $dx(p)\neq 0$ for all $p\in Z$, so we can express $\omega$ near $X_S$ as 
\begin{equation}\label{eq:normal}\omega_0=  \sum_{i=1}^k\dfrac{dx_i}{x_i}\wedge\dfrac{dy_i}{y_i}+\sum_{j=1}^{m}\dfrac{dz_j}{z_j}\wedge\alpha_j+\delta.\end{equation} where  $\alpha_j\in\Omega^1(X_S)$ is a closed form representing  $b_{z_j}$, and $\delta\in\Omega^2(X_S)$ is a closed form representing $a$.  

\begin{proposition}\label{prop01} Let $\omega$ by a partitionable $\log$ symplectic form on a manifold $(M,D)$. Let  $S$  be any subset of $D$ such that $\displaystyle{X_S=\bigcap_{Z\in S} Z}$ is a \emph{maximal intersection} and let $\Lambda_S$ be a subpartition of $S$. If $(a,b_1,\dots,b_k,c_{1,2},\dots,c_{k-1,k})$ is a decomposition of $[\omega]\in {}^{b}H^2(M)|_{X_S}$, then there exist \begin{itemize}\item hypersurface defining functions $x_i$, $y_i$, $z_i$ partitioned according to $\Lambda_S$, \item a tubular neighborhood $U\supset X_S$, and \item  $\alpha_j\in\Omega^1(X_S)$ a closed form representing  $b_{z_j}$ and $\delta\in\Omega^2(X_S)$ a closed form representing $a$, \end{itemize}
such that on $U$ there is a $\log$-symplectomorphism pulling $\omega$ back to (\ref{eq:normal}).
\end{proposition} 

\begin{proof} 

Let $\omega$ be a partitionable log symplectic form on $(M,D)$. Given a maximal intersection $X_S$ given by $S\subseteq D$, let $(a,b_1,\dots,b_k,c_{1,2},\dots,c_{k-1,k})$ be a decomposition of $[\omega]\in {}^{b}H^2(M)|_{X_S}$. Let $x_i$, $y_i$, $z_i$ be hypersurface defining functions  partitioned according to $\Lambda_S$.  
By the isomorphism from equation (\ref{eq:logcohom}), in these coordinates $\omega|_{X_S}=$ $$\sum_{i,j}\left(\dfrac{dx_i}{x_i}\wedge\left(a_{ij}\dfrac{dx_j}{x_j}+b_{ij}\dfrac{dy_j}{y_j}+c_{ij}\dfrac{dz_j}{z_j}\right)+\dfrac{dy_i}{y_i}\wedge\left(d_{ij}\dfrac{dy_j}{y_j}+e_{ij}\dfrac{dz_j}{z_j}\right)\right)$$ $$+\sum_{i,j}f_{ij}\dfrac{dz_i}{z_i}\wedge\dfrac{dz_j}{z_j}+\sum_k\left(\dfrac{dx_k}{x_k}\wedge A_k+\dfrac{dy_y}{y_k}\wedge B_k+\dfrac{dz_k}{z_k}\wedge C_k\right)+\delta.$$

From $d\omega=0$, it follows that $a_{ij},b_{ij},c_{ij},d_{ij},e_{ij},f_{ij}$ are all closed $0$-forms and thus are real numbers. By the definition of partitionable, the only non-zero numbers among these are $b_{ii}$. Further, by the definition of partitionable, the one-forms $A_k$ and $B_k$ satisfy $\displaystyle{A_k|_{Z_{x_k}\cap Z_{y_k}}=B_k|_{Z_{x_k}\cap Z_{y_k}}=0}$. Thus $A_k|_{X_S}=B_k|_{X_S}=0$. 

Thus under an appropriate relabeling and change of $X_S$ defining functions $$\omega|_{X_S}=  \sum_{i=1}^k\dfrac{dx_i}{x_i}\wedge\dfrac{dy_i}{y_i}+\sum_{j=1}^{m}\dfrac{dz_j}{z_j}\wedge\alpha_j+\delta.$$ 

Now we will proceed by the standard relative Moser argument (See \cite{Cannas} Sec. 7.3 for the smooth setting, and \cite{Scott} Thm 6.4 for the $b$-symplectic version).  Pick a tubular neighborhood $U_0$ of $X_S$. Define $\omega_0$ to be $\omega|_{X_S}$ pulled back to $U$. Then $$\omega-\omega_0=\sum_{j=1}^m\dfrac{dz_j}{z_j}\wedge(\alpha_j-\widetilde{\alpha}_j)+\delta-\tilde{\delta}.$$

Since the form $\omega-\omega_0$ is closed on $U_0$, and  $(\omega-\omega_0)|_{X_s}=0$ and $\delta-\tilde{\delta}=0$, by the relative Poincar\'{e} Lemma, there exist primitives $\mu_j$ of $\alpha_j-\widetilde{\alpha}_j$ and a primitive $\sigma$ of $\delta-\tilde{\delta}$ such that $\mu_j|_{X_s}=\sigma|_{X_S}=0$. Define $$\mu=\sum_{j=1}^m\dfrac{dz_j}{z_j}\mu_j+\sigma.$$ 

Then $\omega-\omega_0=d\mu$. Let $\omega_t=(1-t)\omega_0+t\omega$. Then $\dfrac{d\omega_t}{d t}=\omega-\omega_0=d\mu$.  Because $\mu$ is a log one form, the vector field defined by $i_{v_t}\omega_t=-\mu$ is a log vector field and its flow fixes the divisor $D$. Further, $v_t=0$ on $X_s$. Thus we can integrate $v_t$ to an isotopy that is the identity on $X_{S}$ and fixes $D$. This isotopy is the desired $\log$-symplectomorphism. \end{proof} 

This proposition gives us $k$-cosymplectic structures on every intersection of subsets of $D$: 

\begin{proposition}\label{inducedstructure}
Let $\omega$ be a partitionable $\log$ symplectic form $\omega$ on a manifold $(M,D)$. For any set $S\subseteq D$ such that $\displaystyle{\bigcap_{z\in S}Z}$ is nonempty, let $\displaystyle{X= \bigcap_{z\in S}Z}$ away from higher order intersections. The form $\omega$ induces a $\ell$-cosymplectic structure on $X$ where $\ell$ is the number of type $z$ forms in $\Lambda_S$. 
\end{proposition}

\begin{proof} For any $S\subseteq D$ such that $\displaystyle{\bigcap_{z\in S}Z}$ is nonempty, let $\displaystyle{X= \bigcap_{z\in S}Z}$ away from higher order intersections. By equation (\ref{eq:normal}), $$\omega|_
X=  \sum_{i=1}^k \dfrac{dr_i}{r_i}\wedge\dfrac{ds_i}{s_i}+\sum_{j=1}^m\dfrac{dt_j}{t_j}\wedge\theta_j+\beta$$ for closed $\theta_j\in\Omega^1(X)$ and closed $\beta\in\Omega^2(X)$. 
By the non-degeneracy of $\omega$, $$0\neq \wedge^n\omega|_
X=\big(\bigwedge_i\dfrac{dr_i}{r_i}\wedge\dfrac{ds_i}{s_i}\big)\wedge\big(\bigwedge_j \dfrac{dt_j}{t_j}\wedge\theta_j\big)\wedge\beta^{n-k-m}.$$ Thus $$\big(\bigwedge_j \theta_j\big)\wedge\beta^{n-k-m}\neq 0$$ and $(\theta_j,\beta)$ is an $\ell$-cosymplectic structure on $X$. 
\end{proof}

By the standard symplectic linear algebra argument (for instance see \cite{Cannas} Sec. 1.1) we can express $\omega$ at point $p\in D$ as $$\omega_p=\sum_{i=1}^k \dfrac{dx_i}{x_i}\wedge\dfrac{dy_i}{y_i}+\sum_{j=1}^m\dfrac{dz_j}{z_j}\wedge ds_j+\sum_{k=1}^n dp_k\wedge dq_k.$$ By the proof of Proposition \ref{prop01}, a relative Moser's argument gives an analogue of Darboux's theorem for partitionable $\log$ symplectic structures. 

\begin{corollary} 
Let $\omega$ be a partitionable $\log$ symplectic form $\omega$ on a manifold $(M,D)$. Let $p\in D$ and let $S$ be the subset of $D$ of hypersurfaces containing $p$. Let $\Lambda_S$ be a subpartition of $S$. Then there exist local coordinates centered at $p$ with local hypersurface defining functions $x_i$, $y_i$, $z_i$ partitioned according to $\Lambda_S$ such that  $$\omega=\sum_{i=1}^k\dfrac{dx_i}{x_i}\wedge\dfrac{dy_i}{y_i}+\sum_{j=1}^m\dfrac{dz_j}{z_j}\wedge ds_j+\sum_{k=1}^ndp_k\wedge dq_k $$ and the Poisson bi-vector associated to $\omega$ has the form $$\pi=\sum_{i =1}^k x_iy_i\dfrac{\partial}{\partial y_i}\wedge\dfrac{\partial}{\partial x_i}+\sum_{j=1}^m z_j \dfrac{\partial}{\partial s_j}\wedge\dfrac{\partial}{\partial z_j}+\sum_{k=1}^n \dfrac{\partial}{\partial q_k}\wedge\dfrac{\partial}{\partial p_k}. $$ \end{corollary} 

\begin{remark}\label{kremark}

In general, $\mathcal{A}$-symplectic structures for a Lie algebroid $\mathcal{A}$ can reduce showing the existence of `Darboux' coordinates for a variety of structures into a standard symplectic linear algebra and a symplectic relative Moser argument. 

Indeed we will next show that every $k$-cosymplectic structure $(\alpha_j,\beta)$ can sit inside a larger manifold $(M,D)$ such that a log symplectic form on $M$ induces $(\alpha_j,\beta)$ on $W$. 
Thus, given a $k$-cosymplectic manifold $(W,\alpha_j,\beta)$, there exist local coordinates
$$(s_1,\dots,s_k,p_1,q_1,\dots,p_n,q_n)$$ such that $$\alpha_j=ds_j\text{ and }\beta=\sum dp_k\wedge dq_k.$$ 
\end{remark} 

\subsection{Examples from products}

In Example 18 of \cite{Guillemin01}, Guillemin, Miranda, and Pires explained how given any compact $b$-symplectic surface $(M_b,\pi_b)$ and any compact symplectic surface $(M_s,\pi_s)$, the product $(M_b\times M_s,\pi_b + \pi_s)$ is a b-Poisson manifold. Similarly, the product of any partitionable log symplectic surfaces will produce a partitionable log symplectic manifold. 

For instance, the torus $\mathbb{T}^2\simeq\mathbb{S}^1_\theta\times\mathbb{S}^1_\rho$ is a log symplectic surface with partitionable form  $$\omega=\dfrac{d\theta}{\sin(\theta)}\wedge\dfrac{d\rho}{\sin(\rho)}.$$ 

In general, we have a product structure for partitionable log symplectic manifolds: Given two partitionable log symplectic manifolds $(M_1,D_1,\omega_1)$ and $(M_2,D_2,\omega_2)$, one can show that the product $$(M_1\times M_2,(D_1\times M_2)\cup (D_2\times M_1),\omega_1+\omega_2)$$  is also a partitionable log symplectic manifold. Further, this product respects the existing partitions of $D_1$ and $D_2$: For instance, if $Z\in D_1$ was type $x$, then $Z\times M_2$ is type $x$ in $(D_1\times M_2)\cup (D_2\times M_1)$. 

We can also explicitly construct partitionable log symplectic manifolds from a given $k$-cosymplectic structure by taking a product with a torus. 
\begin{Example}
Let $(M,\beta)$ be any symplectic manifold. Then $\mathbb{T}^k\times\mathbb{T}^k\times M$ with $\theta_1,\dots,\theta_k,\rho_1,\dots,\rho_k$ angular coordinates on $\mathbb{T}^k\times\mathbb{T}^k$ is a partitionable log symplectic structure with the form $$\omega = \sum_{i=1}^k\dfrac{d\theta_i}{\sin(\theta_i)}\wedge \dfrac{d\rho_i}{\sin(\rho_i)}+\beta.$$ 

Let $(M,\alpha_j,\beta)$ be any $k$-cosymplectic manifold. Then $\mathbb{T}^k\times M$ with $\theta_i$ the angle coordinates on $\mathbb{S}^1$  is a partitionable log symplectic structure with the form $$\omega = \sum_{i=1}^k\dfrac{d\theta_i}{\sin(\theta_i)}\wedge \alpha_i+\beta.$$ 
\end{Example}

 \subsection{Symplectic foliations} Next we will show that the class of partitionable $\log$ symplectic structures induce regular foliations on the intersection of hypersurfaces in $D$. 
 
By Propositions \ref{prop01} and \ref{inducedstructure} there are exactly two types of behavior for a partitionable $\log$ symplectic form at the nonempty intersection of two hypersurfaces $Z_1,Z_2$ in divisor $D$ away from $D\setminus \left\{Z_1,Z_2\right\}$.

\begin{multicols}{2}

{\bf Type 1:} In the first instance, we have a form of the type $$\omega=\dfrac{dx}{x}\wedge\dfrac{dy}{y}+\delta$$ where $X=\left\{x=0\right\}$ and $Y=\left\{y=0\right\}$. Then $X\cap Y$ is a symplectic leaf in the foliation induced by $\omega$. This leaf extends the foliation on $X$ away from $X\cap Y$ and extends the foliation on $Y$ away from $X\cap Y$. 

\columnbreak \begin{center}
{\bf Type 1 Intersection} 

\vspace{.5ex} 
\includegraphics[scale=.8]{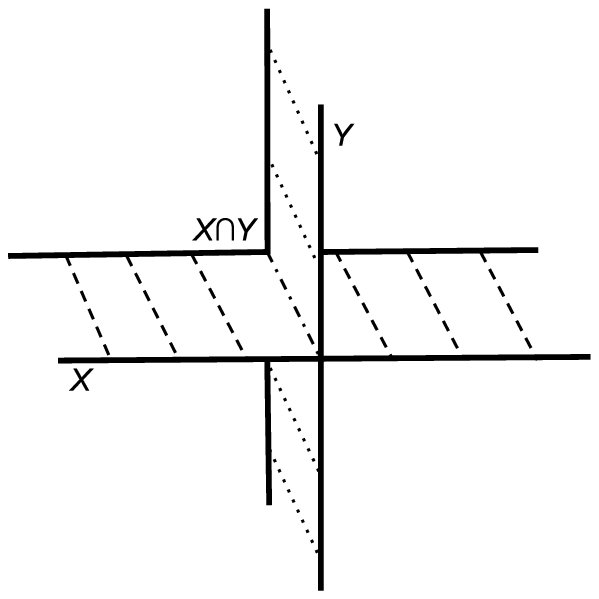} \end{center}

\end{multicols}

\begin{multicols}{2}

{\bf Type 2:} In the second instance, we have a form of the type $$\omega=\dfrac{dz}{z}\wedge\alpha +\dfrac{dv}{v}\wedge\beta+\delta$$ where $Z=\left\{z=0\right\}$ and $V=\left\{v=0\right\}$. Then $\omega$ induces a codimension 2 symplectic foliation on $Z\cap V$.  

\columnbreak \begin{center}

{\bf Type 2 Intersection} 

\vspace{.5ex} 

 \includegraphics[scale=.8]{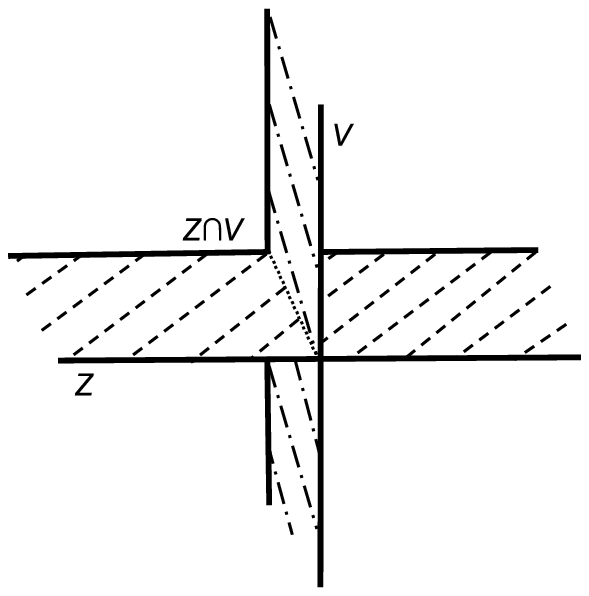}  \end{center}
\end{multicols}

For general intersections, let $I,J,K\subseteq\left\{1,\dots,k\right\}$ and $ L\subseteq\left\{1,\dots,m\right\}$ such that $I\cap J=I\cap K =\emptyset.$ On $$W=\bigcap_{i\in I}\left(Z_{x_i}\cap Z_{y_i}\right)\bigcap_{j\in J}Z_{x_j}\bigcap_{k\in K}Z_{y_k}\bigcap_{\ell\in L}Z_{z_\ell}$$ away from $W\cap (D\setminus \left\{Z_{x_i}, Z_{y_i},Z_{x_j},Z_{y_k},Z_{z_\ell}\right\})$, $\omega$ induces a regular codimension $|J|+|K|+|L|$ symplectic foliation. Further, the foliation is given by the $k$-cosymplectic structure provided in Proposition \ref{inducedstructure}. 

 \section{\bf Proof of Theorem \ref{theorem01}}
 
 Recall, given a Poisson manifold $(M,\pi)$, the Poisson cohomology $H^*_\pi(M)$ is defined as the cohomology groups of the Lichnerowicz complex (see for instance \cite{Dufour}, p. 39). The $k$-th element in the sequence is made up of smooth $k$-multivector fields on $M$, $\mathcal{V}^k(M):=\mathcal{C}^{\infty}(M;\wedge^k TM)$. $$\dots\to\mathcal{V}^{k-1}(M)\xrightarrow{d_{\pi}}\mathcal{V}^{k}(M)\xrightarrow{d_{\pi}}\mathcal{V}^{k+1}(M)\to\dots$$ The differential $$d_{\pi}:  \mathcal{V}^k(M)\to\mathcal{V}^{k+1}(M)$$ is defined as $$d_{\pi}=[\pi,\cdot],$$ where $[\cdot,\cdot]$ is the Schouten bracket extending the standard Lie bracket on vector fields $\mathcal{V}^1(M)$. 
 
 Before delving into the details of the proof of Theorem \ref{theorem01}, we will sketch the computation of Poisson cohomology for a partitionable log symplectic form on $\mathbb{T}^4$ to motivate the constructions involved in the proof.

 \subsubsection{{\bf Motivating Computation.}} Consider $\mathbb{T}^4$ identified as $\mathbb{T}^2\times\mathbb{T}^2$ with angular coordinates $(x,y)$ and $(z,t)$ respectively. Let $D$ be  $$Z_x=\left\{\sin(x)=0\right\}, Z_y=\left\{\sin(y)=0\right\},\text{ and }Z_z=\left\{\sin(z)=0\right\}.$$ We equip this manifold with the log symplectic form $$\omega=\dfrac{dx}{\sin(x)}\wedge\dfrac{dy}{\sin(y)}+\dfrac{dz}{\sin(z)}\wedge dt.$$ This symplectic form is the inverse to the Poisson bi-vector $$\pi=\sin(x)\sin(y)\dfrac{\partial}{\partial y}\wedge\dfrac{\partial}{\partial x}+\sin(z)\dfrac{\partial}{\partial t}\wedge\dfrac{\partial}{\partial z}.$$ 
 
The image $\omega^\flat(TM)$ is spanned by $$\dfrac{dx}{\sin(x)\sin(y)},~\dfrac{dy}{\sin(x)\sin(y)},~\dfrac{dz}{\sin(z)},~ \dfrac{dt}{\sin(z)}.$$
 We can realize this image as the dual of the Lie algebroid, called $\mathcal{R}$, spanned by $$\sin(x)\sin(y)\dfrac{\partial}{\partial x},~ \sin(x)\sin(y)\dfrac{\partial}{\partial y}, ~\sin(z)\dfrac{\partial}{\partial z},~ \sin(z)\dfrac{\partial}{\partial t}.$$ This is a Lie algebroid with anchor map inclusion into $TM$ and with Lie bracket induced by the standard Lie bracket. 
 
 The Lichnerowicz Poisson complex of $\pi$ is isomorphic to the Lie algebroid de Rham cohomology of $\mathcal{R}$.  We compute the de Rham cohomology of  $\mathcal{R}$ in stages, by first computing the de Rham cohomology of a subcomplex $\mathcal{A}$. 
 
 Let $\mathcal{A}$ denote the Lie algebroid spanned by $$\dfrac{\partial}{\partial x},~ \dfrac{\partial}{\partial y},~ \sin(z)\dfrac{\partial}{\partial z},~ \sin(z)\dfrac{\partial}{\partial t}.$$ This vector bundle is a Lie algebroid with anchor map the inclusion into $TM$ and with Lie bracket induced by smoothly extending the standard Lie bracket away from $Z$ to $Z$. 
 
 Because $\mathcal{A}\subseteq \mathcal{R}$ is a sub Lie algebroid, there is an inclusion of Lie algebroid de Rham complexes $${}^{\mathcal{A}}\Omega^p(\mathbb{T}^4)\to{}^{\mathcal{R}}\Omega^p(\mathbb{T}^4).$$
 
 We will first compute ${}^{\mathcal{A}}H^p(\mathbb{T}^4)$.  A degree-p $\mathcal{A}$ form has an expression $$\mu=\dfrac{dz}{\sin(z)}\wedge\dfrac{dt}{\sin(z)}\wedge \cos(z)A+\dfrac{dz}{\sin(z)}\wedge B+\dfrac{dt}{\sin(z)}\wedge C+D$$ where $A\in\Omega^{p-2}(Z_z)$, $B\in\Omega^{p-1}(Z_z)$, $C\in\Omega^{p-1}(Z_z)$, and $ D\in\Omega^p(\mathbb{T}^4)$. Then $d\mu =$ $$\dfrac{dz}{\sin(z)}\wedge\dfrac{dt}{\sin(z)}\wedge \cos(z)dA-\dfrac{dz}{\sin(z)}\wedge dB-\cos(z)\dfrac{dz}{\sin^2(z)}\wedge dt \wedge C-\dfrac{dt}{\sin(z)}\wedge dC+dD.$$  Note that $\cos(z)=\pm 1$ when $\sin(z)=0$. In particular, $\cos(z)=1$ when $z=0$ and $\cos(z)=-1$ when $z=\pi$. Thus $$\ker d=\left\{ C=dA,\hspace{2ex} dB=0, \hspace{2ex}  dD = 0\right\}.$$ 
 If $\mu$ is closed, then there is a degree-$(p-1)$ $\mathcal{A}$ form $$\widetilde{\mu}= \dfrac{dz}{\sin(z)}\wedge b+\dfrac{dt}{\sin(z)}\wedge A+\delta$$ such that \begin{equation}\label{eq:01}\mu+d\widetilde{\mu}=\dfrac{dz}{\sin(z)}\wedge (B-db)+(D+d\delta).\end{equation} 
 
 Thus $[\mu+d\widetilde{\mu}]\in {}^{\mathcal{A}}H^p(\mathbb{T}^4)$ has a representative as in equation (\ref{eq:01}). This computation also shows that if $\mu=\dfrac{dz}{\sin(z)}\wedge B+D\in d({}^{\mathcal{A}}\Omega^{p-1}(\mathbb{T}^4))$, then $\mu=d\nu$ for some $\nu\in{}^{\mathcal{A}}\Omega^{p-1}(\mathbb{T}^4)$ where $$\nu=\dfrac{dz}{\sin(z)}\wedge\dfrac{dt}{\sin(z)}\wedge\cos(z)\alpha +\dfrac{dz}{\sin(z)}\wedge \beta+\dfrac{dt}{\sin(z)}\wedge\gamma +\delta$$ and $B=d\beta$, $D=d\delta$. Thus $B$ and $D$ are exact. Further, if two forms $\nu_1, \nu_2$ are representatives of the same cohomology class in ${}^{\mathcal{A}}H^p(\mathbb{T}^4)$, this shows that the coefficients of the expression $\nu_1-\nu_2$ must be exact. 
 Thus we have shown that $${}^{\mathcal{A}}H^p(\mathbb{T}^4)=\frac{\left\{{B\in\Omega^{p-1}(Z_z):dB=0}\right\}}{\left\{{B:B=db,b\in \Omega^{p-2}(Z_z)}\right\}}\bigoplus \frac{\left\{{D\in\Omega^{p}(\mathbb{T}^4):dD=0}\right\}}{\left\{{D:D=d\delta,\delta\in \Omega^{p-1}(\mathbb{T}^4)}\right\}}$$
 
 and $${}^\mathcal{A}H^p(\mathbb{T}^4)\simeq H^p(\mathbb{T}^4)\oplus H^{p-1}(Z_z).$$
 
 \vspace{2ex} 
 
 Now we can compute ${}^{\mathcal{R}}H^p(\mathbb{T}^4)$ using ${}^{\mathcal{A}}H^p(\mathbb{T}^4)$. A degree-p $\mathcal{R}$ form $\mu$ has an expression $$\dfrac{dx\wedge dy }{\sin^2(x)\sin^2(y)}\wedge (A_{00}+A_{10}\cos(y)\sin(x)+A_{01}\cos(x)\sin(y)+A_{20}\cos(y)\sin^2(x))$$ $$+\dfrac{dx\wedge dy}{\sin^2(x)\sin^2(y)}\wedge (A_{02}\cos(x)\sin^2(y)+A_{11}\sin(x)\sin(y))$$  $$+\dfrac{dx}{\sin(x)\sin(y)}\wedge (B_0+B_1\sin(x))+\dfrac{dy}{\sin(x)\sin(y)}\wedge (C_0+C_1\sin(y))$$ $$+\dfrac{dx}{\sin(x)}\wedge D+\dfrac{dy}{\sin(y)}\wedge E + F$$ where $A_{i,j}\in{}^{\mathcal{A}}\Omega^{p-2}(Z_x\cap Z_y)$, $C_i,D\in{}^{\mathcal{A}}\Omega^{p-1}(Z_x)$, $B_i,E\in{}^{\mathcal{A}}\Omega^{p-1}(Z_y)$, and $ F\in{}^{\mathcal{A}}\Omega^p(\mathbb{T}^4)$. Further, $B_0$ is independent of $x$ and $C_0$ is independent of $y$. 
 
 Then $$d\mu =\dfrac{dx\wedge dy}{\sin^2(x)\sin^2(y)}\wedge (dA_{00}+dA_{10}\cos(y)\sin(x)+dA_{01}\cos(x)\sin(y))$$ 
 
 $$+\dfrac{dx\wedge dy}{\sin^2(x)\sin^2(y)}\wedge (dA_{20}\cos(y)\sin^2(x)+ dA_{02}\cos(x)\sin^2(y)+dA_{11}\sin(x)\sin(y))$$  $$+\cos(y)\dfrac{dx}{\sin(x)}\wedge\dfrac{dy}{\sin^2(y)}\wedge (B_0+B_1\sin(x)) -\dfrac{dx}{\sin(x)\sin(y)}\wedge (dB_0+dB_1\sin(x))$$ $$-\cos(x)\dfrac{dx}{\sin^2(x)}\wedge\dfrac{dy}{\sin(y)}\wedge (C_0+C_1\sin(y)) -\dfrac{dy}{\sin(x)\sin(y)}\wedge (dC_0+dC_1\sin(y))$$ $$-\dfrac{dx}{\sin(x)}\wedge dD-\dfrac{dy}{\sin(y)}\wedge dE + dF.$$  Note that $\cos(x)=\pm 1$ when $\sin(x)=0$ and $\cos(y)=\pm 1$ when $\sin(y)=0$. Thus $\ker d$ is determined by the relations 
 $$dA_{00}=0, \hspace{2ex} B_0=-dA_{10}, \hspace{2ex} B_1=-dA_{20}, \hspace{2ex}
 C_0=dA_{01},$$
 $$C_1=dA_{02}, \hspace{2ex} dA_{11}=0, \hspace{2ex} dD = 0, \hspace{2ex} dE=0, \hspace{2ex} dF=0$$ even though $B_1$ could depend on $x$ and $C_1$ could depend on $y$ above.  
 
\noindent If $d\mu=0$, then there is a degree-$(p-1)$ $\mathcal{R}$ form $\widetilde{\mu}$ of the form $$\dfrac{dx\wedge dy}{\sin^2(x)\sin^2(y)}\wedge (a_{00}+a_{11}\sin(x)\sin(y))+\dfrac{dx}{\sin(x)\sin(y)}\wedge (A_{10}+A_{20}\sin(x))$$ $$+\dfrac{dy}{\sin(x)\sin(y)}\wedge (A_{01}+A_{02}\sin(x))+\dfrac{dx}{\sin(x)}\wedge \delta+\dfrac{dy}{\sin(y)}\wedge e + f$$ such that $[\mu-d\widetilde{\mu}]\in {}^{\mathcal{R}}H^p(\mathbb{T}^4)$ has a representative  $$\dfrac{dx\wedge dy}{\sin^2(x)\sin^2(y)}\wedge ((A_{00}-da_{00})+(A_{11}-da_{11})\sin(x)\sin(y))$$ $$+\dfrac{dx}{\sin(x)}\wedge (D+d\delta)+\dfrac{dy}{\sin(y)}\wedge (E+de) + (F-df).$$ 
 
This computation also shows that if $\mu=$  $$\dfrac{dx\wedge dy}{\sin^2(x)\sin^2(y)}\wedge (A_{00}+A_{11}\sin(x)\sin(y))+\dfrac{dx}{\sin(x)}\wedge D+\dfrac{dy}{\sin(y)}\wedge E + F$$ is in $ d({}^{\mathcal{R}}\Omega^{p-1}(M)),$ then $\mu=d\nu$ for some $\nu\in{}^{\mathcal{A}}\Omega^{p-1}(M)$ where $\nu=$ $$\dfrac{dx\wedge dy}{\sin^2(x)\sin^2(y)}\wedge (a_{00}+a_{10}\cos(y)\sin(x)+a_{01}\cos(x)\sin(y)+a_{20}\cos(y)\sin^2(x))$$ 
$$+\dfrac{dx\wedge dy}{\sin^2(x)\sin^2(y)}\wedge (a_{02}\cos(x)\sin^2(y)+a_{11}\sin(x)\sin(y))$$ $$+\dfrac{dx}{\sin(x)\sin(y)}\wedge (b_0+b_1\sin(x))+\dfrac{dy}{\sin(x)\sin(y)}\wedge (c_0+c_1\sin(y))$$ $$+\dfrac{dx}{\sin(x)}\wedge \delta+\dfrac{dy}{\sin(y)}\wedge e + f$$ and $A_{00}=da_{00}$, $A_{11}=da_{11}$, $D=d\delta$, $E=de$, and $F=df$. Thus if two forms $\nu_1, \nu_2$ are representatives of the same cohomology class in ${}^{\mathcal{R}}H^p(\mathbb{T}^4)$, then the coefficients of the expression $\nu_1-\nu_2$ must be exact.  Thus, we have shown ${}^\mathcal{R}H^p(\mathbb{T}^4)$ is $$\frac{\left\{{A_{00}\in{}^{\mathcal{A}}\Omega^{p-2}(Z_x\cap Z_y):dA_{00}=0}\right\}}{\left\{{A_{00}=da_{00},a_{00}\in {}^{\mathcal{A}}\Omega^{p-3}(Z_x\cap Z_y) }\right\}}\bigoplus\frac{\left\{{A_{11}\in{}^{\mathcal{A}}\Omega^{p-2}(Z_x\cap Z_y):dA_{11}=0}\right\}}{\left\{{A_{11}=da_{11}, a_{11}\in {}^{\mathcal{A}}\Omega^{p-3}(Z_x\cap Z_y) }\right\}}$$  $$\bigoplus\frac{\left\{{D\in{}^{\mathcal{A}}\Omega^{p-1}(Z_x):dD=0}\right\}}{\left\{{D=d\delta,\delta\in {}^{\mathcal{A}}\Omega^{p-2}(Z_x)}\right\}}\bigoplus \frac{\left\{{E\in{}^{\mathcal{A}}\Omega^{p-1}(Z_y):dE=0}\right\}}{\left\{{E=de,e\in {}^{\mathcal{A}}\Omega^{p-2}(Z_y)}\right\}}\bigoplus $$ $$\frac{\left\{{F\in{}^{\mathcal{A}}\Omega^{p}(\mathbb{T}^4):dF=0}\right\}}{\left\{{F=df,f\in {}^{\mathcal{A}}\Omega^{p-1}(\mathbb{T}^4)}\right\}}.$$
 
 Thus in fixed coordinates $x,y,z,t$, ${}^{\mathcal{R}}H^p(\mathbb{T}^4)$ is computable as 
 
 \centerline{$ {}^{\mathcal{A}}H^p(\mathbb{T}^4)\oplus {}^{\mathcal{A}}H^{p-1}(Z_y)\oplus {}^{\mathcal{A}}H^{p-1}(Z_x)\oplus {}^{\mathcal{A}}H^{p-2}(Z_x\cap Z_y)\oplus {}^{\mathcal{A}}H^{p-2}(Z_x\cap Z_y)\simeq$} $$\underbrace{H^p(\mathbb{T}^4)\oplus H^{p-1}(Z_z)}_{{}^{\mathcal{A}}H^p(\mathbb{T}^4)}\oplus  \underbrace{H^{p-1}(Z_y)\oplus H^{p-2}(Z_y\cap Z_z)}_{{}^{\mathcal{A}}H^{p-1}(Z_y)} \oplus  \underbrace{H^{p-1}(Z_x) \oplus  H^{p-2}(Z_x\cap Z_z)}_{{}^{\mathcal{A}}H^{p-1}(Z_x)}$$ 
 $$\oplus \underbrace{H^{p-2}(Z_x\cap Z_y) \oplus H^{p-3}(Z_x\cap Z_y\cap Z_z)}_{{}^{\mathcal{A}}H^{p-2}(Z_x\cap Z_y)}\oplus \underbrace{H^{p-2}(Z_x\cap Z_y)\oplus H^{p-3}(Z_x\cap Z_y\cap Z_z)}_{{}^{\mathcal{A}}H^{p-2}(Z_x\cap Z_y)}.$$
 
 \vspace{2ex}
 
 \noindent We will now discuss the  details necessary to complete the computation for general partitionable log symplectic structures. 
 
 \subsection{Good tubular neighborhoods}
Consider a manifold $M$ with a set $D$ of smooth transversely intersecting hypersurfaces.  For each point $p\in D$, there is a chart $(U, f)$ of $M$ centered at $p$ such that $f(D)$ is a union of a subset of the coordinate hyperplanes in $\mathbb{R}^n$ intersected with $f(D)$. \begin{center} \includegraphics[scale=.4]{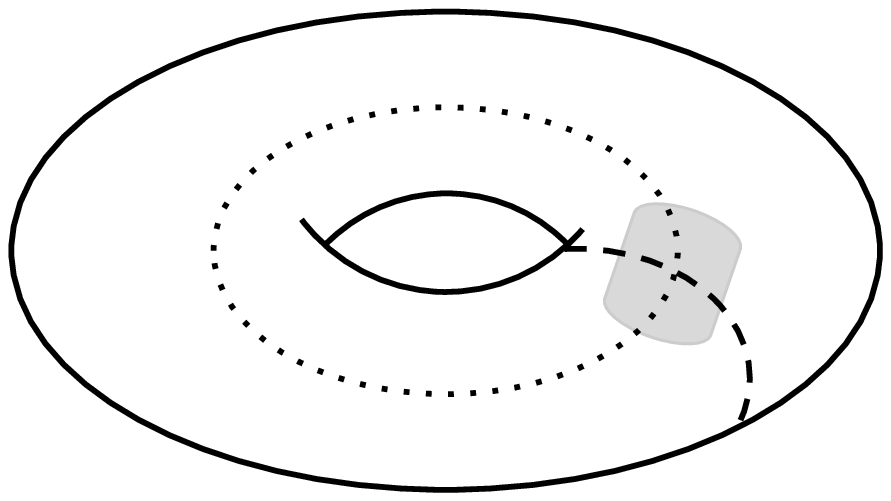}\includegraphics[scale=.7]{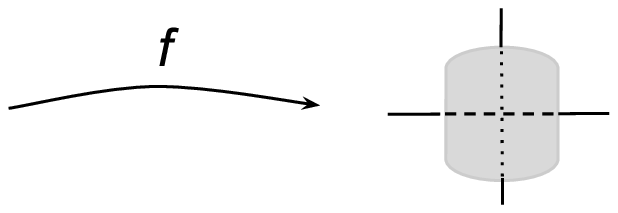}\end{center} 
A good tubular neighborhood $\tau=Z\times(-\varepsilon,\varepsilon)$ of $Z\in D$ is a neighborhood that extends charts of the type $U$ above at the intersection of $Z\cap (D\setminus Z)$. In our computations below we will always use good tubular neighborhoods. For existence of such neighborhoods, see for instance section 5 of \cite{Albin}. 
 
 \subsection{Constructing the rigged algebroid} 
 
To compute the Poisson cohomology of a partitionable $\log$ symplectic manifold we will  use rigged algebroids, see \cite{Lanius} for more details. 

\begin{definition} 
Given a partitionable $\log$ symplectic manifold $(M,D,\omega)$, the dual rigged bundle ${}^{log}\mathcal{R}^*$ is the extension to $M$ of the image $\omega^\flat(TM)$ away from $D$. The log rigged forms are $${}^{log \mathcal{R}}\Omega^p(M)=\mathcal{C}^\infty(M;\wedge^p({}^{log}\mathcal{R}^*)).$$ \end{definition} 

The rigged Lie algebroid is isomorphic to the Poisson Lie algebroid $T^*M$ with anchor map $\pi^\sharp=(\omega^\flat)^{-1}$. Because the $\log$ rigged anchor map $\rho$ is given by inclusion into $TM$, this new setting translates the computation of Poisson cohomology into the familiar language of de Rham cohomology of $T^*M$.  

\begin{center}$\xygraph{!{<0cm,0cm>;<1cm, 0cm>:<0cm,1cm>::}
!{(.25,1)}*+{T^*M}="b"
!{(1.75,1)}*+{{{}^{log}\mathcal{R}}}="c"
!{(1,1)}*+{{\simeq}}="f"
!{(1,0)}*+{TM}="d"
"b":"d"_{\pi^{\sharp}}
"c":"d"^{\rho}}$\end{center} 

Next, we will identify the Lie algebroid ${}^{log}\mathcal{R}$. For the purposes of computing Poisson cohomology, it will be useful to construct ${}^{log}\mathcal{R}$ through a sequence of rescalings. 

\subsubsection{{\bf The Lie algebroid $\mathcal{A}_i$.}} Consider an expression of $\omega$ as in equation (\ref{eq:normal}):  
$$\omega=  \sum_{i=1}^k\dfrac{dx_i}{x_i}\wedge\dfrac{dy_i}{y_i}+\sum_{j=1}^{m}\dfrac{dz_j}{z_j}\wedge\alpha_j+\delta.$$ We will begin by rescaling $TM$ at $Z_{z_1}$ by the vector bundle $\ker\alpha_1\to Z_{z_1}$. In order to employ Theorem 2.2 from \cite{Lanius}, we must verify that $$[\ker\alpha_1,\ker\alpha_1]\subseteq \ker\alpha_1.$$ Let $X,Y\in\ker\alpha_1$. Consider \begin{equation}\label{eq:kercomp} d\alpha_1(X,Y)=X\alpha_1(Y)-Y\alpha_1(X)-\alpha_1([X,Y]).\end{equation} Because $d\alpha_1=0$ and $X,Y\in\ker\alpha_1$, this reduces to $\alpha_1([X,Y])=0$. Thus $[X,Y]\in\ker\alpha_1$. 

Thus by Theorem 2.2 in \cite{Lanius}, there is a Lie algebroid $\mathcal{A}_1$ whose space of sections is $$\left\{u\in\mathcal{C}^\infty(M;TM):u|_{Z_{z_1}}\in\mathcal{C}^\infty(Z_{z_1};\ker\alpha_{1}) \right\}.$$ 

Note that the cotangent bundle $T^*M$ includes into  $\mathcal{A}^*_1$. Thus the one form $\alpha_2\in\Omega^1(M)|_{Z_{z_2}}$ can be lifted to a one form $\widetilde{\alpha}_2=i(\alpha_2)\in{}^{\mathcal{A}_1}\Omega^1(M)|_{Z_{z_2}}$. 

\begin{center}$\xymatrix{
\mathcal{A}^*_1|_{Z_{z_2}}   & \\
T^*M|_{Z_{z_2}} \ar[u]^{i}  & M  \ar@{.>}[lu]_{\widetilde{\alpha}_2} \ar[l]_{\alpha_2}}$\end{center}

Note that we can always lift forms in this way, however the lifted form may vanish at $Z_{z_1}\cap Z_{z_2}$ while the original does not. In order to employ Theorem 2.2 from \cite{Lanius}, we must verify that $\ker\widetilde{\alpha}_2$ is a subbundle of $\mathcal{A}_1|_{Z_{z_2}}$.
 Because  $\alpha_2$ is a closed one-form in a $k$-cosymplectic structure at $Z_{z_2}$, $\ker\alpha_2$ is a subbundle of $TM|_{Z_{z_2}}$. Away from $Z_{z_1}\cap Z_{z_2}$, the inclusion $i$ gives us an isomorphism $\mathcal{A}^*_1|_{Z_{z_2}}\simeq T^*M|_{Z_{z_2}}$ and it is clear that $\ker\widetilde{\alpha}_2$ is a subbundle of $\mathcal{A}^*_1|_{Z_{z_2}}$ . 

Let $p\in Z_{z_1}\cap Z_{z_2}$.  Then, as described in Remark \ref{kremark}, there exist local coordinates $z_1,s_1,z_2,s_2,p_1,\dots,p_n$ of $M$ at $Z_{z_1}\cap Z_{z_2}$ such that $\alpha_1=ds_1$ and $\alpha_2=ds_2$. 
Note that $$T^*_pM\text{ is spanned by }dz_1,ds_1,dz_2,ds_2,dp_1,\dots, dp_n$$
 and $$\mathcal{A}_1^*|_p\text{ is spanned by }\dfrac{dz_1}{z_1},\dfrac{ds_1}{z_1},dz_2,ds_2,dp_1,\dots, dp_n.$$ 
 
Note that the support of $\alpha_2$ is the image of the support of $\widetilde{\alpha}_2$ under the anchor map of $\mathcal{A}_1$.  Thus rank(ker$\alpha_2$)=rank(ker$\widetilde{\alpha}_2$) and $\ker\widetilde{\alpha}_2$ is a subbundle of $\mathcal{A}_1|_{Z_{z_2}}$. 
 
 We will form $\mathcal{A}_2$ by rescaling $\mathcal{A}_1$ by $\ker\widetilde{\alpha}_2$ at $Z_2$. By the computation analogous to equation (\ref{eq:kercomp}) with respect to $\widetilde{\alpha}_2$, there exists a Lie algebroid $\mathcal{A}_2$ whose space of sections is $$\left\{u\in\mathcal{C}^\infty(M;\mathcal{A}_1):u|_{Z_{z_2}}\in\mathcal{C}^\infty(Z_{z_2};\ker\widetilde{\alpha}_2) \right\}.$$

To form $\mathcal{A}_j$ we will rescale $\mathcal{A}_{j-1}$ at $Z_{z_{j}}$. As above, we lift the one form $\alpha_j\in\Omega^1(Z_{z_j})$ to the one form $\widetilde{\alpha}_j=i(\alpha_j)\in{}^{\mathcal{A}_{j-1}}\Omega^1(M)|_{Z_{z_j}}$. Analogous to the argument above, one can verify in local coordinates that $\ker\widetilde{\alpha}_j$ is a subbundle of $\mathcal{A}_{j-1}|_{Z_{z_j}}$. By computation (\ref{eq:kercomp}), there exists a Lie algebroid $\mathcal{A}_j$ whose space of sections is $$\left\{u\in\mathcal{C}^\infty(M;\mathcal{A}_{j-1}):u|_{Z_{z_j}}\in\mathcal{C}^\infty(Z_{z_j};\ker\widetilde{\alpha}_{j}) \right\}.$$

\subsubsection{{\bf The Lie algebroid $\mathcal{B}_i$.}}

Next, we rescale $\mathcal{A}_m$ at $Z_{x_i}$ and $Z_{y_i}$. 

As previously described, we can lift the one form $dx_1 \in\Omega^1(M)|_{Z_{x_1}}$ to the one form $\widetilde{dx}_1=i(dx_1)\in{}^{\mathcal{A}_{m}}\Omega^1(M)|_{Z_{x_1}}$ and we can lift $dy_1\in\Omega^1(M)|_{Z_{y_1}}$ to the one form $\widetilde{dy}_1=i(dy_1)\in{}^{\mathcal{A}_{m}}\Omega^1(M)|_{Z_{y_1}}$. The $\mathcal{A}_m$-one form $\widetilde{dx}_1$ is non-zero at $Z_{x_1}$ and the  $\mathcal{A}_m$-one form $\widetilde{dy}_1$ is non-zero at $Z_{y_1}$. Further, by (\ref{eq:kercomp}) above, $[\ker \widetilde{dx}_1,\ker \widetilde{dx}_1]\subseteq \ker \widetilde{dx}_1$ and  $[\ker \widetilde{dy}_1,\ker \widetilde{dy}_1]\subseteq \ker \widetilde{dy}_1$. Additionally, since we are working in a good tubular neighborhood, $$[\partial_{x_1},\partial_{y_1}]=0.$$ Thus $$[\ker \widetilde{dx}_1\cap\ker \widetilde{dy}_1,\ker \widetilde{dx}_1\cap\ker \widetilde{dy}_1]\subseteq \ker \widetilde{dx}_1\cap\ker \widetilde{dy}_1$$ and by Theorem 2.2 in \cite{Lanius}, there exists a Lie algebroid $\mathcal{B}_1$ whose space of sections is $$\left\{u\in\mathcal{C}^\infty(M;\mathcal{A}_m)\bigg|\begin{array}{c} u|_{Z_{x_1}}\in\mathcal{C}^\infty({Z_{x_1}}, \ker\widetilde{dx}_1)\\  u|_{Z_{y_1}}\in\mathcal{C}^\infty({Z_{y_1}},\ker\widetilde{dy}_1)\end{array} \right\}.$$  

We iteratively form $\mathcal{B}_j$ by rescaling $\mathcal{B}_{j-1}$. First, we lift the one form $dx_j \in\Omega^1(M)|_{Z_{x_j}}$ to the one form $\widetilde{dx}_j=i(dx_j)\in{}^{\mathcal{B}_{j-1}}\Omega^1(M)|_{Z_{x_j}}$ and we lift $dy_j\in\Omega^1(M)|_{Z_{y_j}}$ to the one form $\widetilde{dy}_j=i(dy_j)\in{}^{\mathcal{B}_{j-1}}\Omega^1(M)|_{Z_{y_j}}$. Similar to above, the algebroid $\mathcal{B}_n$ is the vector bundle whose space of sections is $$\left\{u\in\mathcal{C}^\infty(M;\mathcal{B}_{j-1})\bigg|\begin{array}{c} u|_{Z_{x_j}}\in\mathcal{C}^\infty({Z_{x_j}}, \ker\widetilde{dx}_j)\\  u|_{Z_{y_j}}\in\mathcal{C}^\infty({Z_{y_j}},\ker\widetilde{dy}_j)\end{array} \right\}.$$ 

By using our local expression for a partitionable log symplectic form $\omega$, one can check that $\mathcal{B}_k={}^{log}\mathcal{R}$ as a vector bundle and, by the continuity of the standard Lie bracket off of $D$, these are in fact isomorphic as Lie algebroids. If $\displaystyle{W=\bigcap_{Z\in D}Z}$ is non-empty, then for all $p\in W$ there exist local coordinates

\centerline{$(x_1,y_1,\dots,x_k,y_k,\hspace{.5ex}z_1,v_1,\dots,z_\ell,v_\ell,\hspace{.5ex}p_1,q_1\dots,p_m,q_m)$} \noindent such that the sections of ${}^{\log}\mathcal{R}$ are smooth linear combinations of 
$$x_1y_1\dfrac{\partial}{\partial x_1},x_1y_1\dfrac{\partial}{\partial y_1},\dots,x_ky_k\dfrac{\partial}{\partial x_k},x_ky_k\dfrac{\partial}{\partial y_k},$$ 

$$z_1\dfrac{\partial}{\partial z_1},z_1\dfrac{\partial}{ \partial v_1},\dots,z_\ell\dfrac{\partial}{\partial z_\ell},z_\ell\dfrac{\partial}{ \partial v_\ell},\text{ and }\dfrac{\partial}{\partial p_1},\dfrac{\partial}{\partial q_1},\dots,\dfrac{\partial}{\partial p_m},\dfrac{\partial}{\partial q_m}.$$ 
We can locally identify ${}^{log}\mathcal{R}^*$ as the span of $$\dfrac{dx_i}{x_iy_i},\dfrac{dy_i}{x_iy_i}, \dfrac{z_j}{z_j},\dfrac{\alpha_j}{z_j},dp_n,dq_n.$$

 The \emph{$\log$ rigged de-Rham forms} are $${}^{\log\mathcal{R}}\Omega^p(M)= \mathcal{C}^\infty(M; \wedge^p ({}^{\log} \mathcal{R}^*)),$$ smooth sections of the $p$-th exterior power of ${}^{\log}\mathcal{R}^*$. This complex has exterior derivative $d$ given by extending the standard smooth differential on $M\setminus D$ to $M$. 
 
 \subsection{Computing the de Rham cohomology of ${}^{\log} \mathcal{R}$ }

\begin{lemma}The Poisson cohomology of a partitionable $\log$ Poisson manifold $(M,D,\pi)$ is isomorphic to the de Rham cohomology $^{log\mathcal{R}}H^*(M)$ of the $\log$ rigged algebroid ${}^{log}\mathcal{R}$. 
\end{lemma} 

The details of the proof of this lemma can be found in Section 5 of \cite{Lanius}.  Note that we have inclusions of de Rham complexes: $${}^{\mathcal{A}_1}\Omega^*(M)\to \dots\to{}^{\mathcal{A}_m}\Omega^*(M)\to{}^{\mathcal{B}_1}\Omega^*(M)\to\dots\to{}^{\mathcal{B}_k}\Omega^*(M)={}^{log\mathcal{R}}\Omega^*(M).$$

\begin{lemma}\label{Alemma} Let $\widetilde{D}$ be the subset of $D$ consisting of hypersurfaces labeled $Z_{z_i}$. The Lie algebroid cohomology of $\mathcal{A}_m$ is isomorphic to the de Rham cohomology of the log $\widetilde{D}$ tangent bundle. That is,  $${}^{\mathcal{A}_m}H^p(M)\simeq  H^p(M)\bigoplus_{\tau\in\mathscr{T}}H^{p-|\tau|}\big(\bigcap_{j\in \tau}Z_{z_j}\big)$$ where $\mathscr{T}$ denotes all  of the nonempty subsets of $\left\{1,\dots,|\widetilde{D}|\right\}$.
\end{lemma}

\begin{proof} 
The bundle map $i: \mathcal{A}_{i}\to \mathcal{A}_{i-1}$ constructed in Theorem 2.2 of \cite{Lanius} is an inclusion of Lie algebroids and hence fits into a short exact sequence of complexes \begin{center} $0\to {}^{\mathcal{A}_{i-1}}\Omega^p(M)\to{}^{\mathcal{A}_i}\Omega^p(M)\to\mathscr{C}^p\to 0$ \end{center}where  $$\mathscr{C}^p={}^{\mathcal{A}_{i}}\Omega^p(M)/{}^{\mathcal{A}_{i-1}}\Omega^p(M).$$ The differential on $^{\mathscr{C}}d$ is induced by the differential $^{\mathcal{A}_i}d$ on ${}^{\mathcal{A}_i}\Omega^{p}(M)$: if $P$ is the projection ${}^{\mathcal{A}_{i}}\Omega^p(M)\to{}^{\mathcal{A}_{i}}\Omega^p(M)/{}^{\mathcal{A}_{i-1}}\Omega^p(M),$ then  $^{\mathscr{C}}d(\eta)=P({}^{\mathcal{A}_i}d(\theta))$ where  $\theta\in{}^{\mathcal{A}_i}\Omega^{p}(M)$ is any form such that $P(\theta)=\eta$. Hence $({}^\mathscr{C}d)^2=0$ and $(\mathscr{C}^*,{}^\mathscr{C}d)$ is in fact a complex. 

Given a good tubular neighborhood $\tau=Z_{z_i}\times(-\varepsilon,\varepsilon)_{z_i}$ of $M$ near $Z_{z_i}$, note that $z_i$ defines a trivialization $t_{z_i}:N^*Z_{z_i}\to\mathbb{R}$ of $N^*Z_{z_i}$. 

We can write a degree-$p$ $\mathcal{A}_i$ form $\mu\in {}^{\mathcal{A}_i}\Omega^p(M)$ as $$\mu = \theta+ \dfrac{dz_i\wedge\alpha_i}{z_i^2}\wedge A + \dfrac{dz_i}{z_i}\wedge B +\dfrac{\alpha_i}{z_i}\wedge C$$ and $\theta\in {}^{\mathcal{A}_{i-1}}\Omega^p(M)$, $A,B,C\in {}^{\mathcal{A}_{i-1}}\Omega^*(Z_{z_i})\simeq  {}^{\mathcal{A}_{i-1}}\Omega^*(Z_{z_i};|N^*Z_{z_i}|^{-1})\text{ by }t_{{z_i}^*}.$

We write $\mathscr{R}(\mu)=\theta$ and $\mathscr{S}(\mu)=\mu-\mathscr{R}(\mu)$ for `regular' and `singular' parts. It is easy to see that $\mathscr{R}(d\mu)=d(\mathscr{R}(\mu))$ and $\mathscr{S}(d\mu)=d(\mathscr{S}(\mu))$. Thus the trivialization $\tau$ induces a splitting ${}^{\mathcal{A}_i}\Omega^*(M)={}^{\mathcal{A}_{i-1}}\Omega^*(M)\oplus\mathscr{C}^*$ as complexes. As a consequence ${}^{\mathcal{A}_i}H^p(M)={}^{\mathcal{A}_{i-1}}H^p(M)\oplus  H^p(\mathscr{C}^{*})$ and we are left to compute the cohomology of the quotient complex.  

After identifying $\mathscr{C}^p=\left\{\mu\in{}^{\mathcal{A}_i}\Omega^p(M):\theta=0\right\}$, the differential is given by 
 $$d\mu= \dfrac{dz_i\wedge\alpha_i}{z_i^2}\wedge (dA-C) - \dfrac{dz_i}{z_i}\wedge dB -\dfrac{\alpha_i}{z_i}\wedge dC.$$ 

Thus $\ker(d:\mathscr{C}^{p}\to\mathscr{C}^{p+1})=\left\{C=dA,\hspace{2ex} dB=0\right\}$. If $d\mu=0$, then there is $$\tilde{\mu}=\dfrac{\alpha_i}{z_i}\wedge A-\dfrac{dz_i}{z_i}\wedge b\in\mathscr{C}^{p-1}$$ such that $d\left(-\dfrac{\alpha_i}{z_i}\wedge A-\dfrac{dz_i}{z_i}\wedge b\right)=\dfrac{dz_i\wedge\alpha_i}{z_i^2}\wedge A +\dfrac{\alpha_i}{z_i}\wedge dA-\dfrac{dz_i}{z_i}\wedge db$. Then $[\mu-d\tilde{\mu}]\in H^p(\mathscr{C})$ has representative $$\dfrac{dz_i}{z_i}\wedge(B-db).$$ 

This computation also shows that if $\mu=\dfrac{dz_i}{z_i}\wedge B\in d(\mathscr{C}^{p-1})$, then $\mu=d\nu$ for some $\nu\in\mathscr{C}^{p-1}$ where $$\nu=\dfrac{dz_i}{z_i}\wedge\dfrac{\alpha_i}{z_i}\wedge a +\dfrac{dz_i}{z_i}\wedge b +\dfrac{\alpha_i}{z_i}\wedge c $$ and $B=d b$. Thus $B$ is exact. Further, if two forms $\nu_1, \nu_2$ are representatives of the same cohomology class in $H^p(\mathscr{C})$, this shows that the coefficients of the expression $\nu_1-\nu_2$ must be exact. 
 
Let us consider the effect of changing the $Z_{z_i}$-defining function. By the change of $Z_{z_i}$ defining function computation found in Example 2.11 of \cite{Lanius}, the cohomology class $[B-db]$ is unambiguous despite a representative being expressed
using a particular $Z_{z_i}$ defining function. 

Thus $\displaystyle{H^p(\mathscr{C})=  \frac{\left\{{B\in{}^{\mathcal{A}_{i-1}}\Omega^{p-1}(M):dD=0}\right\}}{\left\{{B=db,b\in {}^{\mathcal{A}_{i-1}}\Omega^{p-2}(M)}\right\}}}$  
and  $${}^{\mathcal{A}_i}H^p(M)\simeq{}^{\mathcal{A}_{i-1}}H^p(M)\oplus {}^{\mathcal{A}_{i-1}}H^{p-1}(Z_{z_i}).$$

 Since ${}^{\mathcal{A}_1}H^p(M)\simeq H^p(M)\oplus H^{p-1}(Z_{z_1})$, we have that $${}^{\mathcal{A}_m}H^p(M)\simeq  H^p(M)\bigoplus_{\tau\in\mathscr{T}}H^{p-|\tau|}\big(\bigcap_{j\in \tau}Z_{z_j}\big)$$ where $\mathscr{T}$ denotes all  of the nonempty subsets of $\left\{1,\dots,|\widetilde{D}|\right\}$. 
\end{proof}

\begin{lemma}The Lie algebroid cohomology of $\mathcal{B}_m$ is isomorphic to $${}^{b}H^p(M)\oplus\bigoplus_{\mathscr{M}}H^{p-m}(\bigcap\underbrace{Z_{x_i}\cap Z_{y_i}\cap Z_{x_j}\cap Z_{y_k}\cap Z_{z_\ell}}_{i\in I,~j\in J,~ k\in K,~\ell\in L};\otimes\underbrace{|N_{v_i}^*Z_{x_i}|^{-1}\otimes|N_{v_i}^*Z_{y_i}|^{-1}}_{i\in I})$$ where $\mathscr{M}$ denotes all collections of sets $I,J,K,L$ satisfying $$I,J,K\subseteq\left\{1,\dots,k\right\}, L\subseteq\left\{1,\dots,n\right\}\text{ such that } I\neq\emptyset,\text{ and } I\cap J=I\cap K =\emptyset$$ with $m:=2|I|+|J|+|L|+|K|$ and $v_i=Z_{x_i}\cap  Z_{y_i}$. \end{lemma}
\begin{proof} 
We set $\mathcal{B}_0=\mathcal{A}_m$. The bundle map $i: \mathcal{B}_{i}\to \mathcal{B}_{i-1}$ constructed in Theorem 2.2 of \cite{Lanius} is an inclusion of Lie algebroids and hence fits into a short exact sequence of complexes 
 $$0\to{}^{\mathcal{B}_{i-1}}\Omega^p(M)\to{}^{\mathcal{B}_i}\Omega^p(M)\to\mathscr{C}^p\to 0.$$
where  $$\mathscr{C}^p={}^{\mathcal{B}_{i}}\Omega^p(M)/{}^{\mathcal{B}_{i-1}}\Omega^p(M).$$ The differential on $^{\mathscr{C}}d$ is induced by the differential $^{\mathcal{B}_i}d$ on ${}^{\mathcal{B}_i}\Omega^{p}(M)$: if $P$ is the projection ${}^{\mathcal{B}_{i}}\Omega^p(M)\to{}^{\mathcal{B}_{i}}\Omega^p(M)/{}^{\mathcal{B}_{i-1}}\Omega^p(M),$ then  $^{\mathscr{C}}d(\eta)=P({}^{\mathcal{B}_i}d(\theta))$ where  $\theta\in{}^{\mathcal{B}_i}\Omega^{p}(M)$ is any form such that $P(\theta)=\eta$. Hence $({}^\mathscr{C}d)^2=0$ and $(\mathscr{C}^*,{}^\mathscr{C}d)$ is in fact a complex. 

Given good tubular neighborhoods $\tau_x=Z_{x_i}\times(-\varepsilon,\varepsilon)_{x_i}$ of $M$ near $Z_{x_i}$  and $\tau_y=Z_{y_i}\times(-\varepsilon,\varepsilon)_{y_i}$ of $M$ near $Z_{y_i}$, note that $z_{x_i}$ defines a trivialization $t_{x_i}:N^*Z_{x_i}\to\mathbb{R}$ of $N^*Z_{x_i}$ and $z_{y_i}$ defines a trivialization $t_{y_i}:N^*Z_{y_i}\to\mathbb{R}$ of $N^*Z_{y_i}$. 

We can write a degree-$p$ $\mathcal{B}_i$ form $\mu\in {}^{\mathcal{B}_i}\Omega^p(M)$ as $$\mu =\theta+\dfrac{dx_i\wedge dy_i}{x_i^2y_i^2}\wedge (A_{00}+A_{10}x_i+A_{01}y_i+A_{20}x^2+A_{02}y^2+A_{11}x_iy_i)$$ $$+\dfrac{dx_i}{x_iy_i}\wedge (B_0+B_1x_i)+\dfrac{dy_i}{x_iy_i}\wedge (C_0+C_1y_i)+\dfrac{dx_i}{x_i}\wedge D+\dfrac{dy_i}{y_i}\wedge E$$ where $\theta\in {}^{\mathcal{B}_{i-1}}\Omega^p(M)$, $$A_{k,l}\in {}^{\mathcal{B}_{i-1}}\Omega^*(Z_{x_i}\cap Z_{y_i})\simeq  {}^{\mathcal{B}_{i-1}}\Omega^*(Z_{x_i}\cap Z_{y_i};|N_{v_i}^*Z_{x_i}|^{-1}\otimes |N_{v_i}^*Z_{x_i}|^{-1} )$$ by $t_{{x_i}^*}$, $t_{{y_i}^*}$ where $v_i=Z_{x_i}\cap Z_{y_i}$, $$C_i,D\in {}^{\mathcal{B}_{i-1}}\Omega^*(Z_{x_i})\simeq  {}^{\mathcal{B}_{i-1}}\Omega^*(Z_{x_i};|N^*Z_{x_i}|^{-1})\text{ by }t_{{y_i}^*}, \text{ and }$$ $$B_i,E\in {}^{\mathcal{B}_{i-1}}\Omega^*(Z_{y_i})\simeq  {}^{\mathcal{B}_{i-1}}\Omega^*(Z_{y_i};|N^*Z_{y_i}|^{-1})\text{ by }t_{{x_i}^*}.$$ Further, $B_0$ is independent of $x$ and $C_0$ is independent of $y$. 

We write $\mathscr{R}(\mu)=\theta$ and $\mathscr{S}(\mu)=\mu-\mathscr{R}(\mu)$ for `regular' and `singular' parts. It is easy to see that $\mathscr{R}(d\mu)=d(\mathscr{R}(\mu))$ and $\mathscr{S}(d\mu)=d(\mathscr{S}(\mu))$. Thus the trivializations $\tau_{x_i},\tau_{y_i}$ induce a splitting ${}^{\mathcal{B}_i}\Omega^*(M)={}^{\mathcal{B}_{i-1}}\Omega^*(M)\oplus\mathscr{C}^*$ as complexes. As a consequence ${}^{\mathcal{B}_i}H^p(M)={}^{\mathcal{B}_{i-1}}H^p(M)\oplus  H^p(\mathscr{C}^{*})$ and we are left to compute the cohomology of the quotient complex.  

After identifying $\mathscr{C}^p=\left\{\mu\in{}^{\mathcal{B}_i}\Omega^p(M):\theta=0\right\}$, the differential is given by $d\mu =$ $$\dfrac{dx_i\wedge dy_i}{x_i^2y_i^2}\wedge (dA_{00}+(dA_{10}+B_0)x_i+(dA_{01}-C_0)y_i+(dA_{20}+B_1)x^2+(dA_{02}-C_1)y^2)$$ 
$$+\dfrac{dx_i\wedge dy_i}{x_i^2y_i^2}\wedge dA_{11}x_iy_i-\dfrac{dx_i}{x_iy_i}\wedge dB -\dfrac{dy_i}{x_iy_i}\wedge dC-\dfrac{dx_i}{x_i}\wedge dD-\dfrac{dy_i}{y_i}\wedge dE.$$ Thus $ker(d:\mathscr{C}^p\to\mathscr{C}^{p+1})$ can be identified with the relations $$dA_{00}=0,\hspace{2ex} B_0=-dA_{10}, \hspace{2ex} B_1=-dA_{20}, \hspace{2ex} C_0=dA_{01},$$
$$C_1=dA_{20}, \hspace{2ex} dA_{11}=0, \hspace{2ex} dD=0, \hspace{2ex} dE=0$$ even though $B_1$ could depend on $x_i$ and $C_1$ could depend on $y_i$ above. 

If $d\mu=0$, then there is an element $\tilde{\mu}$ in $\mathscr{C}^{p-1}$, $$\tilde{\mu}=-\dfrac{dx_i}{x_iy_i}\wedge (A_{10}+A_{20}x_i)+\dfrac{dy_i}{x_iy_i}\wedge (A_{01}+A_{02}y_i)+\dfrac{dx_i\wedge dy_i}{x_i^2y_i^2}(a_{00}+a_{11}x_iy_i)$$ $$+\dfrac{dx_i}{x_i}\wedge\delta+\dfrac{dy_i}{y_i}\wedge e$$ such that  $\mu-d\tilde{\mu}$ equals $$\dfrac{dx_i\wedge dy_i}{x_i^2y_i^2}\wedge \bigg((A_{00}-da_{00})+(A_{11}-da_{11})x_iy_i\bigg)+\dfrac{dx_i}{x_i}\wedge(D+d\delta)+\dfrac{dy_i}{y_i}\wedge(E+de).$$ 

This computation also shows that if   $$\mu=\dfrac{dx_i\wedge dy_i}{x^2_iy^2_i}\wedge (A_{00}+A_{11}x_iy_i)+\dfrac{dx_i}{x_i}\wedge D+\dfrac{dy_i}{y_i}\wedge E$$ is in $ d(\mathscr{C}^{p-1}),$ then $\mu=d\nu$ for some $\nu\in\mathscr{C}^{p-1}$ where $$\nu=\dfrac{dx_i\wedge dy_i}{x_i^2y_i^2}\wedge (a_{00}+a_{10}x_i+a_{01}y_i+a_{20}x_i^2+a_{02}y_i^2+a_{11}x_iy_i)$$ 
$$ +\dfrac{dx_i}{x_iy_i}\wedge (b_0+b_1x_i)+\dfrac{dy_i}{x_iy_i}\wedge (c_0+c_1y_i)+\dfrac{dx_i}{x_i}\wedge \delta+\dfrac{dy_i}{y_i}\wedge e$$ and $A_{00}=da_{00}$, $A_{11}=da_{11}$, $D=d\delta$, and $E=de$. Thus if two forms $\nu_1, \nu_2$ are representatives of the same cohomology class in $H^p(\mathscr{C}^p)$, then the coefficients of the expression $\nu_1-\nu_2$ must be exact.

Note that $\dfrac{dx_i\wedge dy_i}{x_iy_i}\wedge (A_{11}-da_{11})+\dfrac{dx_i}{x_i}\wedge(D+d\delta)+\dfrac{dy_i}{y_i}\wedge(E+de)$ is a $\log \widetilde{D}_i$ symplectic form for $\widetilde{D}_i=\left\{Z_{z_1},\dots,Z_{z_\ell},Z_{x_1},Z_{y_1},\dots,Z_{x_j},Z_{y_j}\right\}$. Further, ${}^{B_i}\Omega^p(M)$ splits as ${}^{\log \widetilde{D}_i}\Omega^p(M)\oplus\mathscr{D}^p$ for the appropriate quotient complex $\mathscr{D}^p$. Thus we know that the representatives $ (A_{11}-da_{11}), (D+d\delta)$, and $(E+de)$ are invariant under change of $Z_{x_i}$ and $Z_{y_i}$ defining function. 

Thus we are left to compute what happens to $$\dfrac{dx_i\wedge dy_i}{x_i^2y_i^2}\wedge (A_{00}-da_{00})$$ under change of $Z_{x_i}$ and $Z_{y_i}$ defining function. By this computation, which occurs at the conclusion of the proof of theorem 2.15 in \cite{Lanius}, $$\dfrac{dx_i\wedge dy_i}{x_i^2y_i^2}\wedge (A_{00}-da_{00})$$ can be identified as an element of  $${}^{\mathcal{B}_{i-1}} H^{p-2}(Z_{x_i}\cap Z_{y_i})\simeq {}^{\mathcal{B}_{i-1}} H^{p-2}(Z_{x_i}\cap Z_{y_i};|N_{v_i}^*Z_{x_i}|^{-1}\otimes |N_{v_i}^*Z_{y_i}|^{-1})$$ trivialized by $t_{{x_i}^*}$ and $t_{{y_i}^*}$, and where $v_i=Z_{x_i}\cap Z_{y_i}$. Thus $${}^{\mathcal{B}_{i}}H^p(M)\simeq{}^{\mathcal{B}_{i-1}} H^p(M)\oplus{}^{\mathcal{B}_{i-1}} H^{p-1}(Z_{x_i})\oplus{}^{\mathcal{B}_{i-1}} H^{p-1}(Z_{y_i})\oplus{}^{\mathcal{B}_{i-1}} H^{p-2}(Z_{x_i}\cap Z_{y_i})$$ $$\oplus {}^{\mathcal{B}_{i-1}} H^{p-2}(Z_{x_i}\cap Z_{y_i};|N^*Z_{x_i}|^{-1}\otimes |N^*Z_{y_i}|^{-1}).$$ 

 Since $\mathcal{A}_m=\mathcal{B}_0$, by using Lemma \ref{Alemma} we can conclude that the cohomology  ${}^{\mathcal{B}_m}H^p(M)$ is
 $${}^{b}H^p(M)\oplus\bigoplus_{\mathscr{M}}H^{p-m}(\bigcap\underbrace{Z_{x_i}\cap Z_{y_i}\cap Z_{x_j}\cap Z_{y_k}\cap Z_{z_\ell}}_{i\in I,~j\in J,~ k\in K,~\ell\in L};\otimes\underbrace{|N_{v_i}^*Z_{x_i}|^{-1}\otimes|N_{v_i}^*Z_{y_i}|^{-1}}_{i\in I})$$  where $\mathscr{M}$ denotes all collections of sets $I,J,K,L$ satisfying $$I,J,K\subseteq\left\{1,\dots,k\right\}, L\subseteq\left\{1,\dots,n\right\}\text{ such that } I\neq\emptyset,\text{ and } I\cap J=I\cap K =\emptyset$$ with $m:=2|I|+|J|+|L|+|K|$ and $v_i$ denotes $Z_{x_i}\cap  Z_{y_i}$. \end{proof} 
 
 Since $\mathcal{B}_m={}^{\log}\mathcal{R}$, we have reached the conclusion of Theorem \ref{theorem01}.

\end{document}